\documentclass[11pt]{amsart}
\def \co{\colon\!}
\def \R{\mathbb{R}}
\def \Z{\mathbb{Z}}
\usepackage{amscd}
\makeindex

\usepackage{graphicx}

\def\id{\protect\operatorname{id}}

\newtheorem{theorem}{Theorem}
\newtheorem{lemma}[theorem]{Lemma}
\newtheorem{corollary}[theorem]{Corollary}
\newtheorem{conjecture}[theorem]{Conjecture}

\theoremstyle{definition}
\newtheorem{definition}[theorem]{Definition}

\theoremstyle{remark}
\newtheorem{remark}[theorem]{Remark}

\title{Virtual Legendrian isotopy}
\author[V.~Chernov]{Vladimir Chernov}
\address{Department of Mathematics, 6188 Kemeny Hall, Dartmouth College, Hanover, NH 03755, USA}
\email{Vladimir.Chernov@dartmouth.edu}
\author[R.~Sadykov]{Rustam Sadykov}
\address{Departamento de Matem\'aticas,
Cinvestav-IPN
Av. Instituto Polit\'ecnico Nacional 2508
Col. San Pedro Zacatenco
M\'exico, D.F., C.P.~07360, M\'exico}
\email{rstsdk@gmail.com}
\date{\today}

\begin{document}
\begin{abstract} 
An \emph{elementary  stabilization} of a Legendrian knot $L$ in the spherical cotangent bundle $ST^*M$ of a surface $M$ is a surgery that results in attaching a handle to $M$ along two discs away  from the image in $M$ of the projection of the knot $L$. A virtual Legendrian isotopy is a composition of stabilizations, destabilizations and Legendrian isotopies. A class of virtual Legendrian isotopy is called a virtual Legendrian knot.  

In contrast to Legendrian knots, virtual Legendrian knots enjoy the property that there is a 
bijective correspondence between the virtual Legendrian knots and the equivalence classes of Gauss diagrams. 

We study virtual Legendrian knots and show that every such class contains a unique irreducible representative. In particular we get a solution to the following conjecture of Cahn, Levi and the first author: two Legendrian knots in $ST^*S^2$ that are isotopic as virtual Legendrian knots must be Legendrian isotopic in $ST^*S^2.$
\end{abstract}
\maketitle
\leftline {\em \Small 2010 Mathematics Subject Classification.
Primary: 53D10, 57R17 Secondry 57M27}


\leftline{\em \Small Keywords: Legendrian links, stable isotopy, virtual Legendrian links}

\section{Introduction}

Let $M$ be a closed oriented surface, possibly non-connected, and $L$ a Legendrian link in the total space of the spherical cotangent bundle 
$\pi:ST^*M\to M$ of $M$. 
An \emph{elementary  stabilization} of  $L$  is a surgery that results in cutting out from $M$ two discs away from the image $\pi L$ of the projection of $L$ to $M$, and attaching a handle to $M$ along the created boundary components.  The converse operation is called an \emph{elementary destabilization}. More precisely, let $A$ be a simple connected closed curve in $M$ in the complement to $\pi L$. An \emph{elementary destabilization} of $L$ along $A$ consists of cutting $M$ open along $A$ and then capping the resulting boundary circles.  

An elementary destabilization of a link is \emph{trivial} if it chops off a sphere containing no components of $L$.
We say that a Legendrian link is \emph{irreducible} if it does not allow non-trivial destabilizations.

A \emph{virtual Legendrian isotopy}~\cite{CahnLevi} is a composition  of elementary stabilizations, destabilizations, and Legendrian isotopies. A virtual Legendrian isotopy class of a Legendrian link (respectively, Legendrian knot) is called a \emph{virtual Legendrian link} (respectively, \emph{virtual Legendrian knot}). In contrast to Legendrian knots, virtual Legendrian knots enjoy the property \cite{CahnLevi} that there is a bijective correspondence between the virtual Legendrian knots and equivalence classes of Gauss diagrams\footnote{
Similar to Kauffman's~\cite{Kauffman} theory of ordinary virtual knots, the theory of virtual Legendrian knots can be reformulated in three equivalent ways:
\begin{enumerate}
\item As the theory of Legendrian knots in $ST^*F$ modulo stabilization, destabilization and Legendrian isotopy. 
\item As the theory of virtual front diagrams on $\mathbb R^2$ modulo the standard front moves and the virtual front moves, see~\cite[Sections 2 and 7]{CahnLevi}.
\item As the theory of front Gauss diagrams modulo the modifications of Gauss diagrams, see~\cite[Sections 4 and 7]{CahnLevi}. Note that not every Gauss diagram corresponds to an ordinary Legendrian knot. 
\end{enumerate}}.

\begin{figure}[h]
\centering
\includegraphics[height=1.5in]{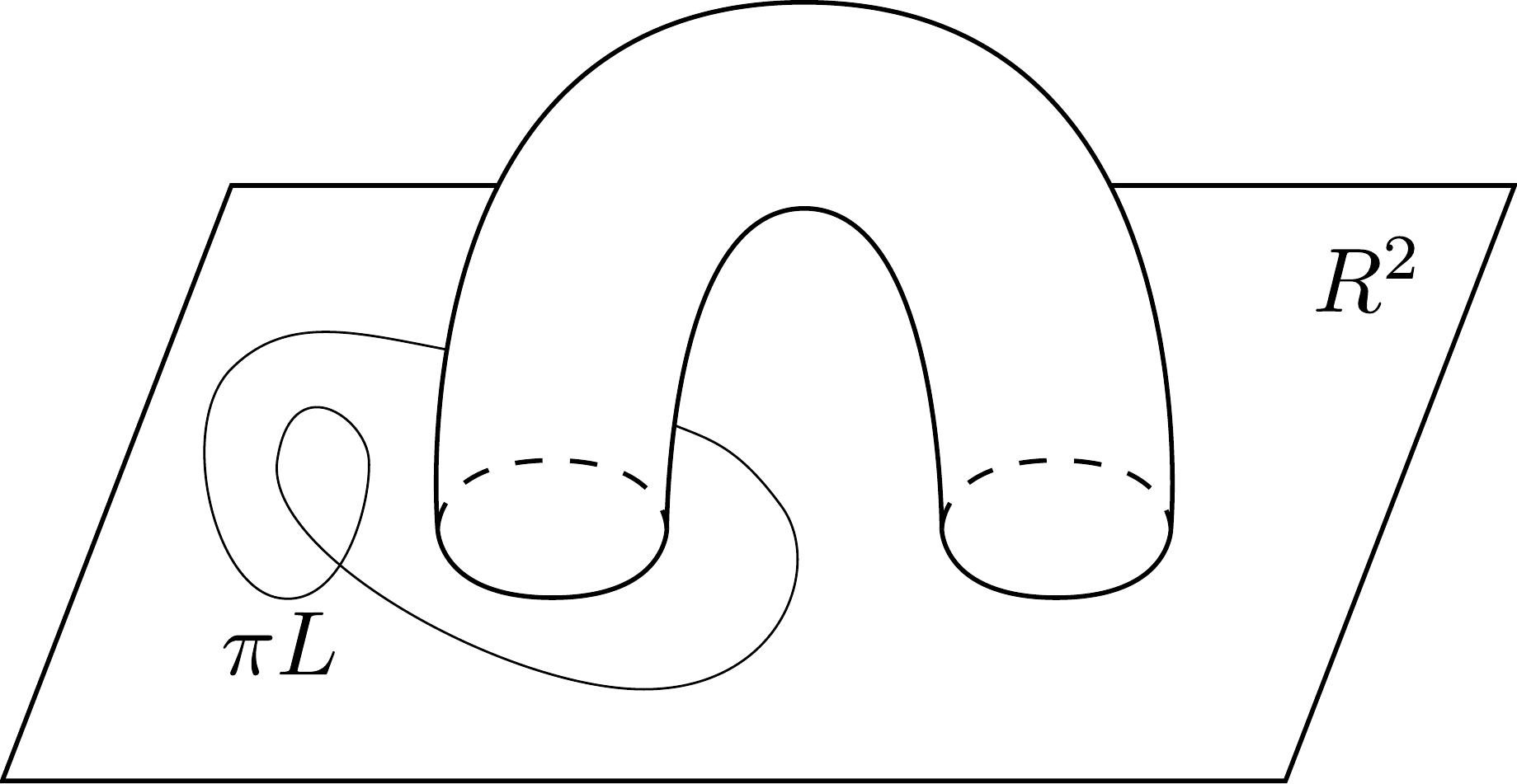}
\caption{An elementary stabilization of a Legendrian curve in the spherical cotangent bundle of $\R^2$.}
\label{fig:1}
\end{figure}



Our main result is Theorem~\ref{th:1} which should be compared to the Kuperberg Theorem on virtual links, \cite[Theorem 1]{GK}. Note that the proof of Kuperberg's results does not seem to generalize to the category of virtual Legendrian knots.

\begin{theorem} \label{th:1} Every virtual isotopy class of Legendrian links contains a unique irreducible representative. The irreducible representative can be obtained from any representative of the virtual Legendrian isotopy class by a composition of destabilizations and isotopies. 
\end{theorem}

The second assertion of Theorem~\ref{th:1} is immediate. Indeed, for any Legendrian link in the given virtual Legendrian isotopy class, only finitely many consecutive non-trivial destabilizations are possible. Thus, after finitely many destabilizations we obtain an irreducible representative.

The main consequence of Theorem~\ref{th:1} is Corollary~\ref{th:2}. 

\begin{corollary}\label{th:2}
Virtual Legendrian isotopy classes of irreducible Legendrian links in $ST^*M$
of a surface $M$ are in bijective correspondence with isotopy classes of irreducible Legendrian links in $ST^*M$. 
\end{corollary}

Note that in general virtual Legendrian isotopy classes of (reducible) links are not in bijective correspondence with Legendrian isotopy classes of links. 

In view of Corollary~\ref{th:2}, we get the solutions to the following two Conjectures~\ref{conjecture1},~\ref{conjecture2} formulated by P.~Cahn, A.~Levi and the first author~\cite[Conjecture 1.4, Conjecture 1.5]{CahnLevi}.

\begin{conjecture}\label{conjecture2}
Let $K_1$ and $K_2$ be two Legendrian knots in $ST^*M$ that are isotopic as virtual Legendrian knots  and suppose that $M$ is the surface of smallest genus realizing knots in the virtual Legendrian isotopy class of $K_1$ and $K_2$. Then possibly after a contactomorphism of $ST^*M$ $K_1$ and $K_2$ are Legendrian isotopic in $ST^*M.$
\end{conjecture}

\begin{conjecture}\label{conjecture1} Two Legendrian knots in $ST^*S^2$ that are isotopic as virtual Legendrian knots must be Legendrian isotopic in $ST^*S^2.$
\end{conjecture}

In~\cite[Conjecture 1.4]{CahnLevi} and \cite[page 5]{CahnLevi} a similar fact is also conjectured for virtual Legendrian knots in 
$ST^*S^n, n\geq 3$ and  $ST^*\R^n, n\geq 2$. These conjectures are still open.

In view of \cite{CahnLevi}, another immediate corollary of Theorem~\ref{th:1} is Corollary~\ref{c:Gauss}. 

\begin{corollary}\label{c:Gauss} Every Gauss diagram can be represented by a unique irreducible Legendrian knot in $ST^*M$ for some surface $M$. 
\end{corollary}

The proof of Theorem~\ref{th:1} consists of several steps, and besides the general case there are two exceptional ones that do not fit the general setting. In section~\ref{s:3} we deal with the exceptional cases. In section~\ref{s:4} we list the steps; these are Lemmas~\ref{l:16}, \ref{maincor} and \ref{equalL}. Lemma~\ref{equalL} is proved in section~\ref{s:4}, while Lemmas~\ref{l:16} and \ref{maincor} are postponed until section~\ref{s:th6} after we present a necessary auxiliary construction.

\subsection*{Acknowledgments}
This paper was written when the first author was at Max Planck Institute for Mathematics, Bonn and the author thanks the institute for its hospitality. This work was partially supported by a grant from the Simons Foundation
$\#235674$ to Vladimir Chernov and by a grant CONACYT $\#179823$ to Rustam Sadykov.

\section{A reformulation of Theorem~\ref{th:1}} 

We say that two links $L_1$ in $ST^*M_1$ and $L_2$ in $ST^*M_2$ are \emph{descent-equivalent}  if after a composition of destabilizations and isotopies of $L_1$ and $L_2$ they become the same.   

Suppose, contrary to the statement of Theorem~\ref{th:1}, there is a Legendrian link $L$ in $ST^*M$ that has two different irreducible descendants. Then there are two Legendrian links $L_1$ and $L_2$ in $ST^*M$, both isotopic to $L$, and two simple closed connected curves $A_1$ and $A_2$ in $M$ such that:
\begin{itemize} 
\item each $A_i$ is disjoint from  $\pi L_i$, and
\item the elementary destabilizations of $L_1$ along $A_1$ and of $L_2$ along $A_2$ are not descent-equivalent. 
\end{itemize}

Note that the second condition implies that  both destabilizations are non-trivial. 
Furthermore, without loss of generality we may assume that 
\begin{itemize}
\item the surface $M$ has no naked sphere components, and 
\item the genus of $M$ is minimal (among those surfaces for which there exist $L_1, L_2, A_1, A_2$ as above).
\end{itemize} 

In particular, every link obtained from $L$ by an elementary non-trivial destabilization (or a composition of elementary non-trivial destabilizations) has a unique descendent. 

Theorem~\ref{th:1} is equivalent to the statement that a tuple $(M, L_1, L_2, A_1, A_2)$ as above does not exist. In the following sections we will assume that such a tuple exists and arrive to a contradiction. 


\section{Exceptional cases}\label{s:3}

We will often require that the manifold $M$ is distinct from a sphere, and that neither $A_1$ nor $A_2$ bounds  a disc; our general argument does not work in these exceptional cases, see Remark~\ref{r:15} below. 

In this section we show that the two assumptions that $M\ne S^2$ and that $A_1$ and $A_2$ are non-contractible are not restrictive (Lemmas~\ref{l:6} and \ref{l:5}).

\begin{lemma}\label{l:6} Suppose that $A_1$ bounds a disc. Then the destabilization of $L_1$ along $A_1$ is descent-equivalent to the destabilization of $L_2$ along $A_2$. 
\end{lemma}
\begin{proof} We will show that we can assume that the intersection of $A_1$ and $A_2$ is empty and hence destabilizations along $A_1$ and $A_2$ are descent equivalent. In more detail, we assume that $\pi L_1$ is located close to the center of the disk $D$ bounded by $A_1$; the case where $\pi L_1$ is located outside $D$ is entirely similar. If $A_1$ and $A_2$ intersect, take such a pair of curves with the minimal number of intersection points among those pairs of curves destabilizations along which  are not descent equivalent.

We show that the number of intersection points may be further decreased yielding a contradiction. $A_2$ subdivides the disk $D$ into regions, 
by induction at least two of these regions are bigons and one of them does not contain the center of $D$ (with $\pi L_1$ in its small neighborhood). The latter bigon is bounded by an arc $\alpha_2$ of $A_2$ and $\alpha_1$ of $A_1$, see Figure~\ref{Bigon.fig}.
Since this bigon does not contain any components of $\pi L_1$, we can compress the arc $\alpha_2$ along this bigon in such a way that during the compression it does not intersect $\pi L_1.$ If this bigon contains curves of $\pi L_2$ they will be pushed out through $\alpha_1$ during the compression.
\end{proof}

\begin{figure}[h]\label{Bigon.fig}
\centering
\includegraphics[height=2in]{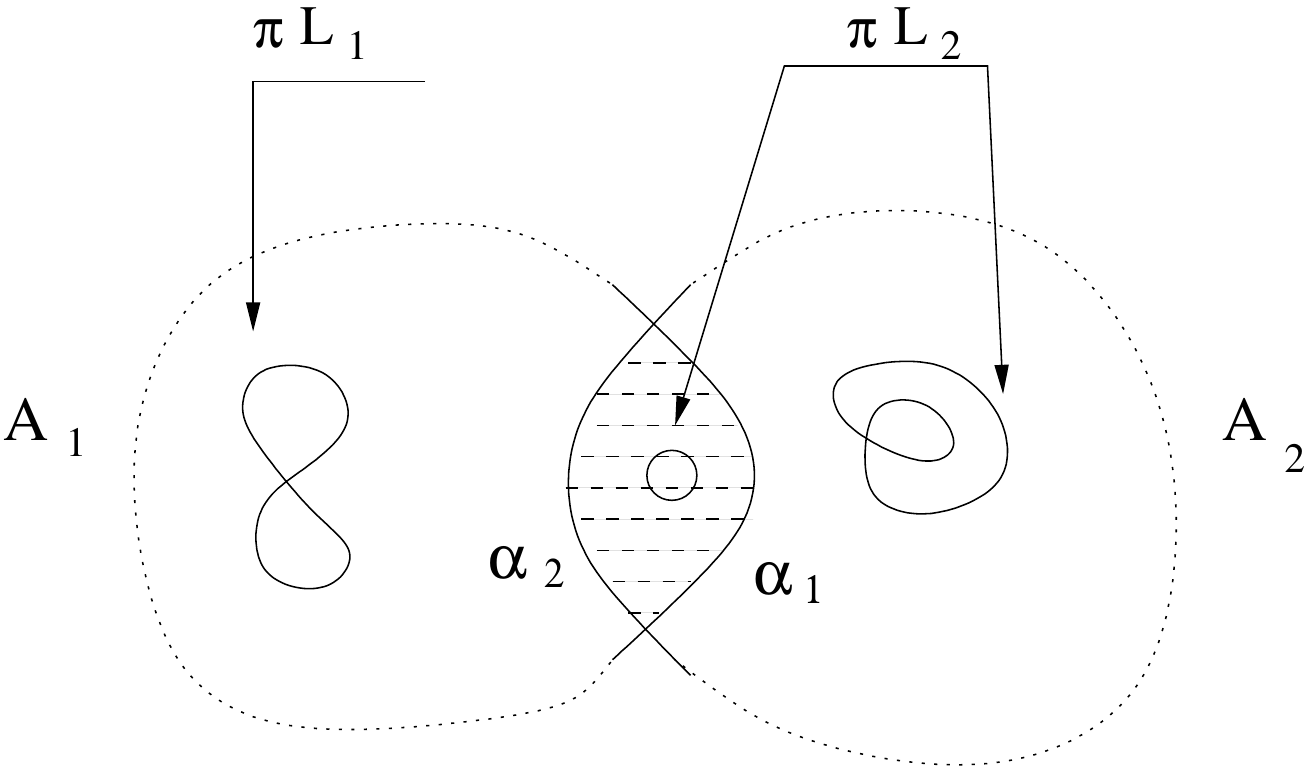}
\caption{The bigon bounded by the arcs $\alpha_1$ and $\alpha_2.$}
\label{fig:1}
\end{figure}

Lemma~\ref{l:5} bellow immediately follows from Lemma~\ref{l:6} since every connected simple closed curve on a sphere bounds a disc. 

\begin{lemma}\label{l:5} The statement of Theorem~\ref{th:1} is true in the case where $M$ is  a sphere. 
\end{lemma}

\section{Proof of Theorem~\ref{th:1}}\label{s:4}

 In view of Lemmas~\ref{l:5} and \ref{l:6}, we may (and will) assume that $A_1$ and $A_2$ are not contractible, and that the surface $M$ is not homeomorphic to a sphere.

\begin{proof}[Proof of Theorem~\ref{th:1}]

Recall that each $\pi L_i$ is disjoint from $A_i$. The following Lemma essentially asserts that we may also assume that $\pi L_1$ is disjoint from both $A_1$ and $A_2$.

\begin{lemma}\label{l:16}  
Suppose that $A_1, A_2$ are not null-homotopic and that the surface $M$ is distinct from a sphere. 
Then there is an isotopy of $L_1$ whose projection does not intersect $A_1$ and that takes $L_1$ to a curve whose projections is disjoint from $A_2$. 
\end{lemma}

The proof of Lemma~\ref{l:16} is postponed till section~\ref{s:th6}. Next, we show that not only we can assume that $L_1$ is disjoint from $A_1$ and $A_2$, but we can, in fact, assume that $L_1=L_2$. 

\begin{lemma}\label{maincor} Assume that $M\neq S^2$ and $A_1, A_2$ are not contractible. If $\pi L_1$ does not intersect $A_2$, then there is a Legendrian isotopy of $L_2$ to $L_1$ whose projection to $M$ avoids the curve $A_2$. 
\end{lemma}

The proof of Lemma~\ref{maincor} will also be given in section~\ref{s:th6}. 
Lemma~\ref{equalL} below completes the proof of Theorem~\ref{th:1} since its conclusion contradicts the minimality of the genus of $M$.  

\begin{lemma}\label{equalL}  If $L_1=L_2=L$, then the genus of $M$ is not minimal. \\
\end{lemma}
\begin{proof}  The  argument is similar to that by Greg Kuperberg. Namely, assume, contrary to the statement, that $L_1=L_2$ and the genus of $M$ is minimal. 

It follows that the intersection $A_1\cap A_2$ is non-empty; otherwise destabilizations of $L$ along $A_1$ and $A_2$ are descent-equivalent.  Without loss of generality we may assume that $A_1$ and $A_2$ intersect in the minimal number of points among pairs of simple connected closed curves destabilizations along which are not descent-equivalent. 

If the two curves $A_1$ and $A_2$ intersect at only one point, then take the boundary $A_3$ of a neighborhood of $A_1\cup A_2$. Note that the destabilization along $A_3$ is not trivial; it chops off a naked torus. 
Destabilizations along $A_1$ and $A_3$ are descent equivalent since the curves are disjoint. Similarly for $A_2$ and $A_3$. Therefore destabilizations along $A_1$ and $A_2$ are descent-equivalent, contrary to the assumption. 

Finally, suppose that $A_1$ and $A_2$ have at least two common points. Let $D_1$ be an interval in $A_1$ bounded by two intersection points and containing no other points of $A_2$.  Compress $A_2$ along $D_1$, i.e., remove small arcs of $A_2$ intersecting $A_1$, and then join the two pairs of boundary points of $A_2$ by two new arcs parallel to $D_1$. Then $A_2$ turns into two new connected curves $A_2'$ and $A_2''$ in $M$, see Figure~\ref{Compression.fig}.  The destabilization along at least one of these components, say $A_2'$, is non-trivial. Observe that the destabilizations of $L_2$ along $A_2$ and $A_2'$ are equivalent since both are disjoint from $\pi L_2$ and have no common points (after a small displacement of one of them along a vector field orthogonal to the curve). On the other hand,  the curve $A'_2$  has less intersection points with $A_1$. Therefore destabilizations along $A_1$, $A_2'$ and $A_2$ are descent equivalent. This completes the proof. 
\end{proof}

\begin{figure}[h]
\centering
\includegraphics[height=1in]{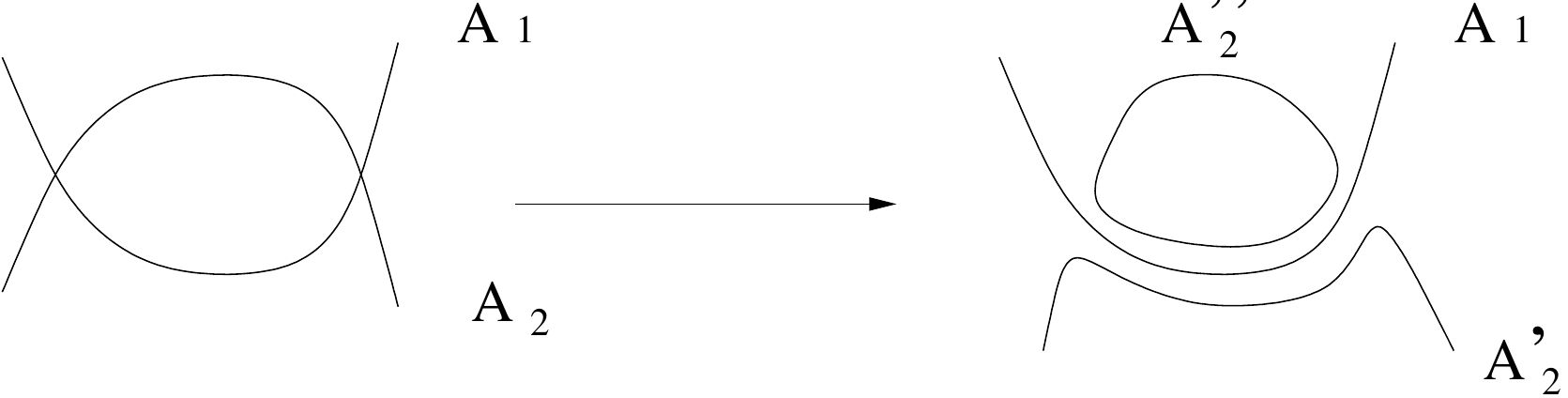}
\caption{Compression of $A_2$ along an arc.}
\label{Compression.fig}
\end{figure}

This completes the proof of Theorem~\ref{th:1} assuming Lemmas~\ref{l:16} and \ref{maincor}. 
\end{proof}

\section{Proof of Lemmas~\ref{l:16} and \ref{maincor}}\label{s:th6}

The main tool in the proof of Lemmas~\ref{l:16} and \ref{maincor} is Theorem~\ref{th:6}. 
To motivate the proof of Theorem~\ref{th:6} let us prove Lemma~\ref{l:10}. We will not use Lemma~\ref{l:10} in what follows. However, this short Lemma~\ref{l:10} explains well the counterintuitive phenomenon  
that stable Legendrian isotopy in certain cases reduces to Legendrian isotopy. 
 
\begin{lemma}\label{l:10} Let $M$ be a hyperbolic surface. Let $L_1$ and $L_2$ be two Legendrian links in $ST^*M$ whose projections belong to an open  disc $D\subset M$. Then $L_1$ and $L_2$ are isotopic in $ST^*M$ if and only if they are isotopic in $ST^*D$. 
\end{lemma}
\begin{proof} Clearly, if $L_1, L_2$ are isotopic in $ST^*D$, then they are isotopic in $ST^*M$. Let us prove the converse implication.

Let $p\co \R^2\to M$ denote the universal covering of $M$. There exist lifts $L_1'$ and $L_2'$ of $L_1$ and $L_2$ respectively such that the isotopy of $L_1$ to $L_2$ lifts to an isotopy of $L_1'$ to $L_2'$ in $ST^*\R^2$. Choose an arbitrary diffeomorphism $\varphi\co \R^2\to D^2$. It lifts to a conctactomorphism $\tilde\varphi$ of spherical cotangent bundles. Thus, we obtain a Legendrian isotopy of $\tilde\varphi(L_1')$ to $\tilde\varphi(L_2')$. It remains to show that $L_1$ admits a Legendrian isotopy  to $\tilde\varphi(L_1')$ and $L_2$ admits a Legendrian isotopy to $\tilde\varphi(L_2')$.

We may assume that both $L_1$ and $L_2$ are links whose images with respect to $\pi$ are located in a small neighborhood $U$ of a point in $D$. Furthermore we may choose $\varphi$ so that the composition of a lift of $D$ and $\varphi$ is the identity map on $U$ so that $\tilde \varphi( L_i)=L_i,i=1,2.$
Then, it remains to show that for any link $L$ in $ST^*D$ and any lifts $L'$ and $L''$  in $ST^*\R^2$, the link $\tilde\varphi(L')$ admits a Legendrian isotopy to  $\tilde\varphi(L'')$. Choose a Legendrian isotopy $\gamma$ from $L'$ to $L''$ in $ST^*\R^2$. The desired Legendrian isotopy is $\tilde\varphi(\gamma)$. 
\end{proof}

\begin{theorem}\label{th:6} Let $L_1$ and $L_2$ be isotopic Legendrian links in the spherical cotangent bundle $ST^*M$ of a connected closed surface $M\ne S^2$, and let $A$ be a simple connected closed curve in $M$ disjoint from $\pi L_1$ and $\pi L_2$. If $A$ breaks $M$ into two surfaces, suppose that $\pi L_1$ and $\pi L_2$ belong to the same path component of $M\setminus A$, and the other path component of $M\setminus A$ is distinct  from the disc.  
Then there exists a Legendrian isotopy of $L_1$ to $L_2$ whose projection to $M$ avoids the curve $A$. 
\end{theorem}

Before proving Theorem~\ref{th:6}, let us construct an (in general, non-regular) covering of $M$ by a surface $\tilde M$ homeomorphic to the connected component of $M\setminus A$ which contains $\pi(L_1)$ and $\pi(L_2).$ 
In fact we will give three equivalent definitions, each has its advantage. 

\begin{definition}[First definition]\label{d:1} Choose a base point in $M$ in the path component of $M\setminus A$ that contains $\pi L_1$ and $\pi L_2$. 
We say that an element in the fundamental group $\pi_1M$ \emph{avoids} $A$ if it admits a representing curve that does not intersect $A$. The subset of elements in $\pi_1M$ avoiding $A$ forms a subgroup. Let $\tilde M\to M$ be the covering corresponding to the subgroup of $\pi_1M$ of elements avoiding $A$. 
\end{definition}

\begin{definition}[Second definition]\label{d:2}  Since $M$ is distinct from a sphere, it admits a universal covering $u\co \R^2\to M$. We choose a base point in $\R^2$ that projects to the base point in $M$. Then  every point $x$ in the universal covering space can be identified with the pair of a point $y=u(x)$ and the homotopy class of the projection in $M$ of the curve in $\R^2$  from the base point  to $x$. The manifold $\tilde M$ is the quotient of $\R^2$ by the relation that identifies $(y, \gamma_1)$ with $(y, \gamma_2)$ whenever $\gamma_1\gamma_2^{-1}$ contain a loop that does not intersect $A$.   
\end{definition}

\begin{definition}[Third definition]\label{d:3} Suppose that $A$ does not separate $M$. Since $M$ is either a torus or hyperbolic, there is a infinite covering $\mathbb{H}\to M$ (or $\mathbb R^2\to M$), and we may assume that a lift $\tilde{A}$ of $A$ is a geodesic (every simple non-contractible curve on $M$ is isotopic to a unique simple geodesic). There is a monodromy action $\mathbb{Z}$ on $\mathbb{H}$ (or on $\mathbb R^2$) corresponding to the loop $A$; namely, we know that $M$ is the quotient of $\mathbb{H}$ (or of $\mathbb R^2$) by the action of $\pi_1M$, and the mentioned monodromy action is the action by the subgroup generated by the loop $A$. It acts on the geodesic $\tilde{A}$ by translations. 
Attach $(\mathbb{H}\setminus \tilde{A})/\Z$ (or $(\mathbb{R}^2\setminus \tilde{A})/\Z$) to $M\setminus A$ so that the projections of the two cylinders $(\mathbb{H}\setminus\tilde{A})/\Z$ (or of $(\mathbb{R}^2\setminus \tilde{A})/\Z$) and of the manifold $M\setminus A$ to $M$ form an infinite covering $\tilde M\to M$; this is the desired covering.

Suppose now that $A$ separates $M$ into two components $M_1$ and $M_2$, where $M_1$ is the component containing the images of the projections of $L_1$ and $L_2$ to $M$. Again, take a covering $\mathbb{H}\to M$ (respectively $\R^2\to M$) and cut $\mathbb{H}$ (respectively $\R^2$) along a lift $\tilde{A}$ of $A$. Attach one component of $(\mathbb{H}\setminus\tilde{A})/\Z$ (respectively of $(\mathbb{R}^2\setminus\tilde{A})/\Z$)
to $M_1$ so that their projections to $M$ form a desired covering $\tilde M\to M$. 
\end{definition}

\begin{remark}\label{r:15} If $A$ bounds a disc, i.e., the case that we exclude from the consideration, then the first and the second definitions result in the one sheet covering, while the third definition makes no sense since a lift of a contractible curve $A$ is not a geodesic. 
\end{remark}

Let $M, A$ be as  in Theorem~\ref{th:6} and $\tilde M\to M$ be the covering from the definitions~\ref{d:1},~\ref{d:2},~\ref{d:3}.
If $A$ does not separate $M$, let $M_1$ denote $M\setminus A$. If $A$ does separate $M$, let $M_1$ denote the connected component of $M\setminus A$ that contains the projections of $L_1, L_2.$

\begin{lemma}\label{l:9}
The surface $\tilde M$ is homeomorphic to $M_1$. 
\end{lemma}
\begin{proof} The statement of Lemma~\ref{l:9} immediately follows from Definition~\ref{d:3}. Indeed, the manifold $\tilde M$ is obtained from $M_1$ by attaching one or two cylinders depending on wether $M_1$ has one or two ends. 
\end{proof}

To summarize we constructed a covering $\tilde M\to M$ by a surface homeomorphic to $M_1$. 

\begin{proof}[Proof of Theorem~\ref{th:6}]
Since $\pi L_1$ is disjoint from $A$, it lifts to a simple (not connected if $L_1$ is a link) closed Legendrian curve $L_1'$ in $ST^*M_1\subset ST^*\tilde M$.  Furthermore, the Legendrian isotopy of $L_1$ to $L_2$ lifts to a Legendrian isotopy of $L_1'$ to $L_2'$
in $ST^*\tilde M$. 

Let $U$ be a thin neighborhood of the boundary of $M_1$ disjoint from $\pi L_1$ and $\pi L_2.$
Suppose that $L_2'$ belongs to the leaf $ST^*M_1\subset ST^*\tilde M$. Then there is  an isotopy of the identity map of $ST^*\tilde M$ relative to $ST^{*} (M_1\setminus U)$ to a map with image in $ST^* M_1$ and that brings the isotopy of $L_1'$ to $L_2'$ into $ST^*M_1$. This isotopy of $\id_{ST^*\tilde M}$ comes from a deformation retraction $\tilde M\to M_1$ fixing points in  $M_1\setminus U$.

\begin{figure}[h]
\centering
\includegraphics[height=2.8in]{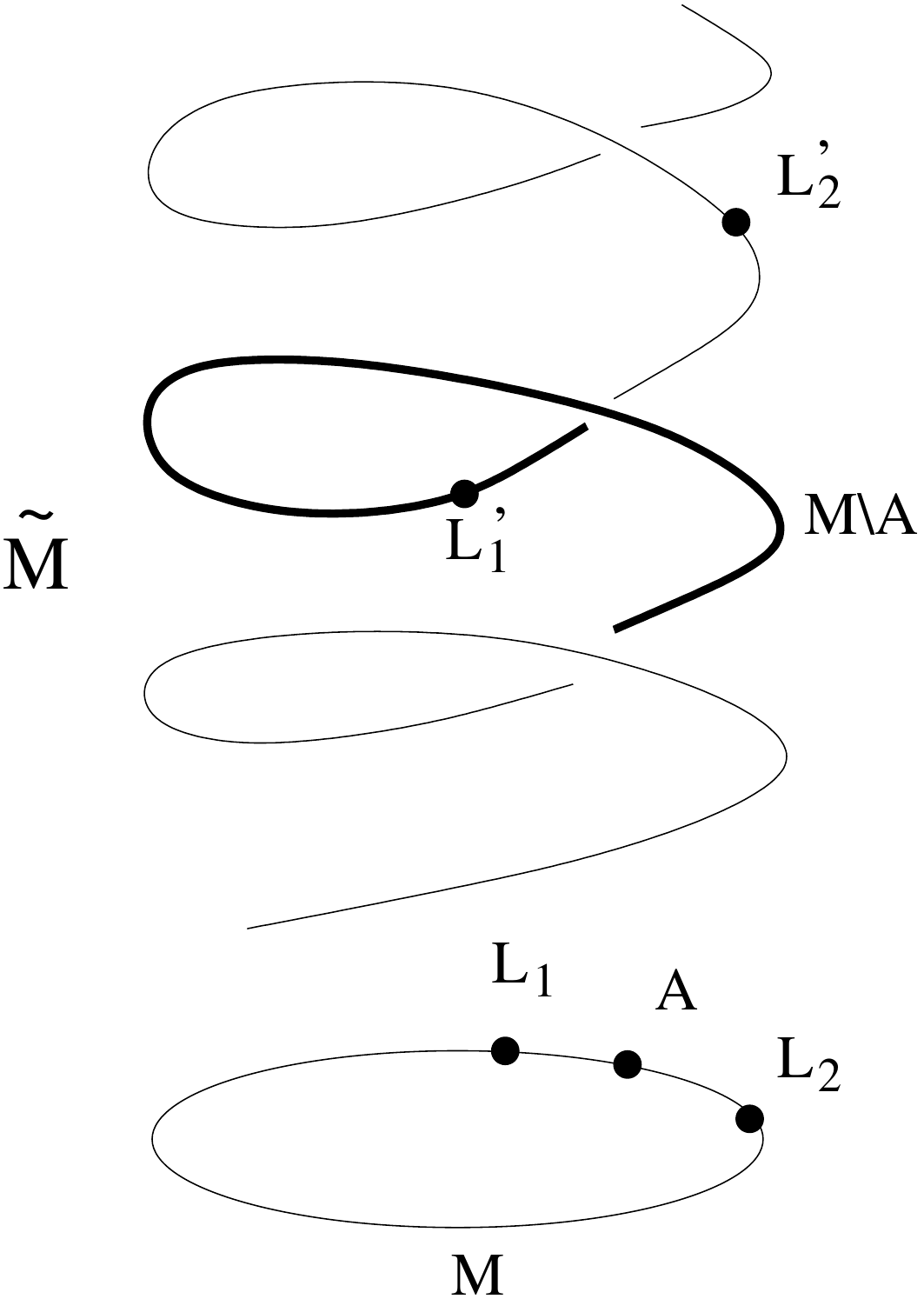}
\caption{Lifts of $L_1$ and $L_2$ to the covering.}
\label{Comvering.fig}
\end{figure}

Suppose now that $L_2'$ belongs to a leaf of the covering $ST^*\tilde M\to ST^*M$ distinct from the leaf $ST^*M_1$. In this case, a deformation retraction of $\tilde M$ to $M_1$ moves $L_2'$ and therefore the above argument does not work. 
During the isotopy of $L'_1$ to $L'_2$ we get that at a certain moment $L'_1$ leaves $ST^*M_1$ and in view of the deformation retraction, in this case we may assume that the projection of $L_1$ to $M$ belongs to the interior of $U\cap M_1$.(Here $U$ is a thin neighborhood of $\partial M_1.$) Similarly, by exchanging the roles of $L_1$ and $L_2$, we may assume that the projection of $L_2$ to $M$ belongs to the interior of $U\cap M_1.$ Now we are in position to apply the isotopy argument. Again, the Legendrian isotopy of $L_1$ to $L_2$ lifts to a Legendrian isotopy of $L_1'$ to $L_2'$. The link $L_2'$ may belong to the leaf distinct from the leaf $ST^*M_1$. However, since its projection $L_2$ is in the interior of $U\cap M_1$, there  is an isotopy that performs a parallel translation of  $L_2'$ to the lift of $L_2$ in the leaf $ST^*M_1$.   This case was considered before.
\end{proof}

\begin{proof}[Proof of Lemma~\ref{maincor}] To simplify notation let us assume that $L$ is a Legendrian knot; the case where $L$ is a link  is similar. If $L_1$ and $L_2$ belong to the same component of $M\setminus A_2$, then the required Legendrian isotopy exists by Theorem~\ref{th:6}.

Suppose now that $A_2$ separates the surface into two components, and that $L_1$
and $L_2$ belong to different path components of $M\setminus A_2$. In this case the argument in the proof of Theorem~\ref{th:6} shows that we may assume that  $\pi L_2$ belongs to a neighborhood of $A_2$. 

Let $L_2'$ be the link obtained from $L_2$ by a translation such that $\pi L_2'$ belongs to the same path component of $M\setminus A_2$ that contains $L_1$. Then the destabilization of $L_2'$ along $A_2$ is descent equivalent to the destabilization of $L_2$ along $A_2$. Indeed, after the destabilization along $A_2$ both $\pi L_2$ and $\pi L_2'$ are curves in a neighborhood of a point, and therefore both links are descent equivalent to the same link in $ST^*S^2$. 

Thus, we may assume that $L_1$ and $L_2$ belong to the same path component of $M\setminus A_2$; this case has been considered above. 
\end{proof}

\begin{proof} [Proof of Lemma~\ref{l:16}] 

Recall that we assumed that $M$ is not a sphere and there is a Legendrian link $L$ represented by a link $L_1$ and $L_2$ and there are two simple closed connected curves $A_1$ and $A_2$ that are not null-homotopic such that the destabilization of $L_1$ along $A_1$ is not descent-equivalent to the destabilization of $L_2$ along $A_2$. Furthermore, we may assume that $A_1$ and $A_2$ are geodesics. Indeed,   there exists an ambient isotopy $\varphi_t$, with $t\in [0,1]$, of the surface $M$ that takes $A_1$ into a geodesic. The ambient isotopy of the surface lifts to an isotopy $\tilde{\varphi}_t$ of the spherical cotangent bundle of $M$. Clearly, the destabilization of the Legendrian link $\tilde{\varphi}_1L_1$ along $\varphi_1A_1$ is descent-equivalent to the destabilization of  $L_1$ along $A_1$. Thus we may assume that $A_1$ is a geodesic. Similarly, we may find an isotopy $\psi_t$ and its lift $\tilde\psi_t$ such that $\psi_1A_2$ is a geodesic, and the destabilization of $L_2$ along $A_2$ is descent-equivalent to the destabilization of $\tilde\psi_1A_2$ along $\psi_1 A_2$. If we replace now the original pairs $(L_1, A_1)$ and $(L_2, A_2)$ with the new pairs $(\tilde\varphi_1L_1, \varphi A_1)$ and $(\tilde\psi_1 L_2, \psi_1 A_2)$, then we obtain an example as the original one but with an additional property that the destabilizations are performed along geodesics.

As in the argument of the proof of Theorem~\ref{th:6}, consider a covering $\tilde M\to M$ by a surface $\tilde M$ homeomorphic to $M_1$. Take the lift of an isotopy from $L_1$ to $L_2$ to a covering isotopy from $L_1'$ to $L_2'$ in $ST^*\tilde M$. A crucial observation is that the inverse image $A_2'$ of $A_2$ in $\tilde M$ consists of disjoint geodesics.  The parts of these geodesics over cylinders attached to $M_1$ are easy to visualize. There is an isotopy of $\tilde M$ to $M_1$ that at each moment of time takes the geodesics of $A_2'$ to themselves. This isotopy takes $L_1$ to a curve disjoint  both from $A_1$ and $A_2$.  
\end{proof}


\begin{thebibliography}{99999}


\bibitem{CahnLevi} {\bf P.~Cahn and A.~Levi:} {\it Vassiliev Invariants of Virtual Legendrian Knots} arXiv:1305.5293 (2013) 25 pages, to appear in the Pacific Journal of Mathematics
\bibitem{Kauffman} {\bf L.~Kauffman:} {\it Virtual knot theory.\/}  European J. Combin.~{\bf 20} (1999), no.~7, 663--690

\bibitem{GK} {\bf G.~Kuperberg:}  {\it What is a virtual link?\/} Algebr.~Geom.~Topol.~{\bf 3} (2003), 587--591



\end{thebibliography}
\end{document}